\documentclass[12pt]{article}

\usepackage{amsmath,amsthm}
\usepackage{amssymb}
\usepackage{enumitem}
\usepackage[T1]{fontenc}
\usepackage{url}

\newtheorem{theorem}{Theorem}[section]
\newtheorem{corollary}[theorem]{Corollary}
\newtheorem{lemma}[theorem]{Lemma}

\newtheorem{conjecture}[theorem]{Conjecture}

\theoremstyle{definition}
\newtheorem{definition}[theorem]{Definition}

\newtheorem{example}[theorem]{Example}

\numberwithin{equation}{section}

\begin{document}

\title{Diophantine Analysis of a Digital Anomaly}

\author{Samer Seraj\\
Existsforall Academy\\ 
Mississauga, Ontario, Canada\\
E-mail: samer\_seraj@outlook.com
}

\date{\today}

\maketitle

\renewcommand{\thefootnote}{}



\renewcommand{\thefootnote}{\arabic{footnote}}
\setcounter{footnote}{0}

\begin{abstract}
The arithmetic-digital anomaly of $5\div 2 = 2.5$ has been observed several times in the past. We generalize it to an exponential Diophantine equation and inequality in the general number base, which is the object of our analysis. First, we produce a near-parametrization of all solutions using a modification of the standard parametrization of Pythagorean triples. We use this parametrized function to find all solutions where the numerator and denominator are coprime, and we construct infinite families where they are not coprime. Next, we use a variant of Baker's theorem from transcendental number theory to prove that each number base admits only finitely many solutions. Lastly, we use the $abc$ conjecture to conditionally show that only finitely many solutions have a numerator with $k$ digits, for each $k\ge 3$. A conjecture is offered for $k=2$.
\end{abstract}
\hfill 

\section{Introduction}

Observe the arithmetic-digital property of the computation, $$\frac{5}{2} = 2.5,$$
where the numerator is the fractional part and the denominator is the integer part of the common value. This curiosity raises questions about which numerator-denominator pairs satisfy the same property, and in which number bases. Formally, we are looking at the exponential Diophantine equation \eqref{eqn:main-anomaly} with its attached inequality constraint \eqref{eqn:unrefined-ineq}.

The problem has been studied previously outside the research setting. For example, a search on the website, \textit{Mathematics Stack Exchange}, yields several relevant threads \cite{Stack1, Stack2, Stack3, Stack4}. However, all of the proposed solutions there are incorrect, or incomplete, or inconclusive about the general case. It is worth noting also that the case of a coprime numerator-denominator pair in base $B=10$ was Problem 1 of Day 1 of the Second Round of the Iran National Mathematical Olympiad in 2013. However, a definitive paper does not seem to have been written. Our aim is to remedy this gap in the literature through original contributions.

\begin{definition}\label{def:digital-anomaly}
The positive integers $(x,y,B,k)$ form a \emph{digital anomaly} if
\begin{equation}\label{eqn:main-anomaly}
    \frac{x}{y} = y + \frac{x}{B^k},
\end{equation}
where $B\ge 2$, and $k$ is the number of digits of $x$ in base $B$, which is equivalent to saying
\begin{equation}\label{eqn:unrefined-ineq}
    B^{k-1} \le x < B^k.
\end{equation}
\end{definition}

\begin{example}
The original motivating example satisfies
$$\frac{5}{2} = 2 + \frac{1}{2} = 2 + \frac{5}{10^1}, \; 10^0 \le 5 < 10^1.$$
\end{example}

To gain perspective, we see that some immediate consequences of Definition~\ref{def:digital-anomaly} are, by contradiction:
\begin{itemize}
    \item $x>y$, because otherwise $x\le y$ would lead to $\frac{x}{y} \le 1 < y + \frac{x}{B^k}$.
    \item $y\nmid x$, because otherwise $y\mid x$ would lead to the left hand side of \eqref{eqn:main-anomaly} being an integer and the right hand side not being an integer.
    \item $y\ge 2$, because otherwise $y=1$ would lead to $y\mid x$, which is a possibility that we already eliminated in the second bullet point.
\end{itemize}

Along the way, our tools will include a modification of the standard parametrization of Pythagorean triples, a variant of Baker's theorem from transcendental number theory, and the $abc$ conjecture. In this way, a seemingly fleeting observation about the digits of the decimal form of a fraction leads to applying a sequence of increasingly deeper ideas from the study of Diophantine analysis.

\section{Parametrizing Digital Anomalies}

\begin{lemma}\label{lem:Pythagorean-equivalence}
The quadruple of positive integers $(x,y,B,k)$ satisfies \eqref{eqn:main-anomaly} if and only if
\begin{equation}\label{eqn:Pythagorean-anomaly}
    (x+2B^k y)^2 + (2B^{2k})^2 = (x+2B^{2k})^2.
\end{equation}
\end{lemma}

\begin{proof}
An expansion of \eqref{eqn:Pythagorean-anomaly} quickly shows that it is equivalent to \eqref{eqn:main-anomaly}. Instead, we display how to devise \eqref{eqn:Pythagorean-anomaly} from \eqref{eqn:main-anomaly} in the first place. The idea is to begin by using the quadratic formula to isolate $y$ in \eqref{eqn:main-anomaly} and then reverse some of the operations (which is the same as completing the square in $y$), followed by completing the square in $x$ as well:
\begin{align*}
    \frac{x}{y} = y + \frac{x}{B^k} &\iff B^k x = B^k y^2 + xy\\
    &\iff B^k y^2 + xy - B^k x = 0\\
    &\iff y = \frac{-x \pm \sqrt{x^2 + 4B^{2k}x}}{2B^k}\\
    &\iff (x+ 2B^k y)^2 = x^2 + 4B^{2k} x\\
    &\iff (x+ 2B^k y)^2 + (2B^{2k})^2 = x^2 + 4B^{2k} x + 4B^{4k}\\
    &\iff (x+ 2B^k y)^2 + (2B^{2k})^2 = (x+2B^{2k})^2.
\end{align*}
\end{proof}

Lemma~\ref{lem:Pythagorean-equivalence} motivates us to recall the classic concept of Pythagorean triples, and their well-known parametrization.

\begin{definition}
A \emph{Pythagorean triple} $(a,b,c)$ consists of three positive integers such that
\begin{equation}\label{eqn:Pythagorean-triple}
    a^2 + b^2 = c^2.
\end{equation}
A Pythagorean triple is said be to \emph{primitive} if $\gcd(a,b,c)=1$.
\end{definition}

Note that $\gcd(a,b,c) = 1$ in a Pythagorean triple $(a,b,c)$ if and only if $\gcd(a,b)=1$ because \eqref{eqn:Pythagorean-triple} implies that any common divisor of $a$ and $b$ is also a divisor of $c$. In particular, a primitive Pythagorean triple cannot allow $a$ and $b$ to both be even. Moreover, considering \eqref{eqn:Pythagorean-triple} modulo $4$, it is known that $a$ and $b$ cannot both be odd. So, $a$ and $b$ must have opposite parities in the primitive case.

\begin{lemma}\label{lem:classic-Pythag-param}
If $(a,b,c)$ is a primitive Pythagorean triple with $b$ even, then there exist uniquely determined positive integers $m$ and $n$ such that
\begin{equation}\label{eqn:classic-Pythag-param}
    (a,b,c) = (m^2 - n^2, 2mn, m^2 + n^2).
\end{equation}
Moreover, it follows that $m>n$, $m$ and $n$ are coprime, and $m$ and $n$ have opposite parities. Conversely, if $m$ and $n$ are coprime positive integers with opposite parities such that $m>n$ and \eqref{eqn:classic-Pythag-param} holds, then $(a,b,c)$ is a primitive Pythagorean triple. \cite[p.~245-246]{Hardy}.
\end{lemma}

\begin{definition}
The \emph{$2$-adic valuation} of a non-zero integer $n$, denoted by $\nu_2 (n)$, is the maximum integer such that $2^{\nu_2 (n)} \mid n$.
\end{definition}

By scaling up primitive Pythagorean triples by positive integer factors $\ell$, Lemma~\ref{lem:classic-Pythag-param} tells us that all Pythagorean triples (not necessarily primitive) with $\nu_2 (a) < \nu_2 (b)$ are uniquely determined as
\begin{equation}\label{eqn:general-Pythag-prelim}
(a,b,c) = (\ell (m^2 - n^2),\ell(2mn), \ell (m^2 + n^2)),
\end{equation}
where $\ell$ ranges over all positive integers, and $m,n$ necessarily satisfy the conditions stated in Lemma~\ref{lem:classic-Pythag-param}.

Although we know that \eqref{eqn:Pythagorean-anomaly} produces the Pythagorean triple
$$(x+2B^k y, 2B^{2k}, x+2B^{2k})$$
the issue is that we do not know which of $x+2B^k y$ and $2B^{2k}$ has a higher $2$-adic valuation. To avoid casework, we will tweak Lemma~\ref{lem:classic-Pythag-param} and \eqref{eqn:general-Pythag-prelim} by accommodating for the possibility that $\nu_2 (a) > \nu_2 (b)$ (assuming $2\mid b$) in a Pythagorean triple $(a,b,c)$, while retaining the structure of the expressions for $(a,b,c)$ in Lemma~\ref{lem:classic-Pythag-param}.

\begin{lemma}\label{lem:flipped-Pythag-param}
Suppose $(a,b,c)$ is a Pythagorean triple such that $\nu_2 (a) > \nu_b (b) \ge 1$. Then, there exist unique positive integers $\ell,m,n$ such that
\begin{equation}\label{eqn:general-Pythag-prelim-2}
(a,b,c) = (\ell (m^2 - n^2),\ell(2mn), \ell (m^2 + n^2)).
\end{equation}
Moreover, it follows that $m>n$, and $m$ and $n$ are both odd and coprime. Conversely, if $\ell$ is a positive integer and $m,n$ are odd, coprime, positive integers with $m>n$, then the triple $(a,b,c)$ in \eqref{eqn:general-Pythag-prelim-2} is a Pythagorean triple satisfying $\nu_2 (a) > \nu_b (b) \ge 1$.
\end{lemma}

\begin{proof}
Suppose $\nu_2 (a) > \nu_b (b) \ge 1$ for a Pythagorean triple $(a,b,c)$. By the comment after Lemma~\ref{lem:classic-Pythag-param}, we know that there exist unique positive integers $\ell_0,m_0,n_0$ such that
$$(a,b,c) = (\ell_0 (2m_0 n_0),\ell_0 (m_0^2 - n_0^2), \ell_0 (m_0^2 + n_0^2)),$$
where $m_0$ and $n_0$ are coprime with opposite parities and $m_0>n_0$. For the construction, let
$$(\ell,m,n) = \left(\frac{\ell_0}{2},m_0+n_0,m_0-n_0\right),$$
where $\ell = \frac{\ell_0}{2}$ is necessarily an integer because $$b = \ell_0 (m_0 - n_0)(m_0+n_0)$$
is even with both $m_0 - n_0$ and $m_0 + n_0$ odd, so $2\mid \ell_0$.
We can check that
\begin{align*}
    \ell(m^2 - n^2) &= \frac{\ell_0}{2}\left[(m_0+n_0)^2-(m_0-n_0)^2\right] = \ell_0(2m_0 n_0) = a,\\
    \ell(2mn) &= \frac{\ell_0}{2}\left[2(m_0+n_0)(m_0-n_0)\right] = \ell(m_0^2 - n_0^2) = b,\\
    \ell(m^2 + n^2) &= \frac{\ell_0}{2}\left[(m_0+n_0)^2+(m_0-n_0)^2\right] = \ell_0(m_0^2 + n_0^2) = c.
\end{align*}
For uniqueness, suppose $(\ell_1, m_1, n_1)$ is another such triple that could replace $(\ell,m,n)$ in the parametrization of $(a,b,c)$. Using the equations
\begin{align*}
    \ell(m^2 - n^2) &= a = \ell_1(m_1^2 - n_1^2),\\
    \ell(2mn) &= b = \ell_1(2m_1 n_1),\\
    \ell(m^2 + n^2) &= c = \ell_1(m_1^2 + n_1^2),
\end{align*}
we get
$$\begin{cases}
    2\ell m^2 = c+a = 2\ell_1 m_1^2\\
    2\ell n^2 = c-a = 2\ell_1 n_1^2
\end{cases}
\implies \frac{m}{n} = \frac{m_1}{n_1}.
$$
Since $\gcd(m,n) = 1 = \gcd(m_1,n_1)$, these fractions are in the unique lowest terms, so $m=m_1$ and $n=n_1$. Then, $$\ell(2mn) = \ell_1(2m_1 n_1)$$ leads to $\ell = \ell_1$.

For the necessary criteria, we check them one by one. Firstly,
$$m = m_0 + n_0 > m_0 - n_0 = n.$$
Since $m_0$ and $n_0$ have opposite parities, $m=m_0+n_0$ and $n=m_0-n_0$ are both odd. Then, the Euclidean algorithm tells us that
\begin{align*}
    \gcd(m,n) &= \gcd(m_0 + n_0, m_0 - n_0) = \gcd(m_0+n_0,2m_0) \\
    &= \gcd(m_0+n_0,m_0) = \gcd(n_0,m_0) = 1.
\end{align*}
Lastly, it is easy to algebraically check that \eqref{eqn:general-Pythag-prelim-2} leads to $a^2 + b^2 = c^2$, with
\begin{align*}
    \nu_2 (a) &= \nu_2 (\ell(m-n)(m+n)) \ge \nu_2 (\ell) + 2\\
    &> \nu_2 (\ell) + 1 = \nu_2 (\ell(2mn)) = \nu_2 (b).
\end{align*}
\end{proof}

By combining the comment after Lemma~\ref{lem:classic-Pythag-param} (for $\nu_2 (a) < \nu_2 (b)$) with Lemma~\ref{lem:flipped-Pythag-param} (for $\nu_2 (a) > \nu_2 (b) \ge 1$), we obtain the following Lemma~\ref{lem:useful-Pythag-param}, which parametrizes all Pythagorean triples for which $2\mid b$, in a manner that retains the coprimality of $m$ and $n$, $m>n$, and the structure of the parametric forms, while being ambiguous about which of $\nu_2 (a)$ and $\nu_2 (b)$ is greater.

\begin{lemma}\label{lem:useful-Pythag-param}
If $(a,b,c)$ is Pythagorean triple (not necessarily primitive), meaning $a^2 + b^2 = c^2$, such that $2\mid b$, then there exist unique positive integers $(\ell,m,n)$ such that
\begin{equation}\label{eqn:useful-Pythag-triple}
    (a,b,c) = (\ell(m^2-n^2),\ell(2mn),\ell(m^2+n^2)).
\end{equation}
Moreover, it follows that $m>n$ and $m,n$ are coprime. Conversely, any triple of positive integers $(\ell,m,n)$ with $m,n$ coprime and $m>n$ leads to the triple $(a,b,c)$, defined by \eqref{eqn:useful-Pythag-triple}, being a Pythagorean triple with $b$ even.
\end{lemma}

Now, Lemma~\ref{lem:useful-Pythag-param} will allow us to uniquely parametrize the set of triples $(x,y,B^k)$ over all digital anomalies $(x,y,B,k)$, thanks to Lemma~\ref{lem:Pythagorean-equivalence}.

\begin{theorem}\label{thm:anomaly-param}
The set of all triples $(x,y,B^k)$, such that $(x,y,B,k)$ is a digital anomaly, is parametrized uniquely by triples of positive integers $(t,m,n)$ such that $m>n$ and $m,n$ are coprime, as
\begin{equation}\label{eqn:digital-param}
    (x,y,B^k) = (tm(m-n)^2, \sqrt{tn}(m-n),m\sqrt{tn}),
\end{equation}
where the parameters must satisfy
\begin{align}
    t(m-n)^4 &< n, \label{eqn:lower-digital-ineq}\\ 
    n^{k-1} &\le m^2 t^{k+1} (m-n)^{4k}. \label{eqn:upper-digital-ineq}
\end{align}
Conversely, any positive integers $(x,y,B,k)$ which satisfy \eqref{eqn:digital-param}, for some coprime $m,n$ such that $m>n$, and that fulfill \eqref{eqn:lower-digital-ineq} and \eqref{eqn:upper-digital-ineq}, will necessarily be a digital anomaly.
\end{theorem}

\begin{proof}
By Lemma~\ref{lem:Pythagorean-equivalence}, $(x,y,B,k)$ is a digital anomaly if and only if
$$(x+2B^k y, 2B^{2k}, x+2B^{2k})$$
is a Pythagorean triple. Since the second entry, $2B^{2k}$, is even, Lemma~\ref{lem:useful-Pythag-param} tells us that there exist unique positive integers $(\ell,m,n)$ such that
\begin{equation}\label{eqn:lmn-to-abc-injection}
    (x+2B^k y, 2B^{2k}, x+2B^{2k}) = (\ell(m^2-n^2),\ell(2mn),\ell(m^2-n^2)),
\end{equation}
where the uniqueness is in correspondence with with values of 
$$(a,b,c) = (x+2B^k y, 2B^{2k}, x+2B^{2k}).$$
Moreover, $m>n$ and $m,n$ are coprime. Solving the system for $x,y,$ and $B^k$ yields
\begin{equation}\label{eqn:preliminary-digital-param}
\begin{aligned}
    x &= (x+2B^{2k}) - 2B^{2k} = \ell(m^2+n^2) - \ell(2mn) = \ell(m-n)^2,\\
    B^k &= \sqrt{\frac{2B^{2k}}{2}} = \sqrt{\frac{\ell(2mn)}{2}} = \sqrt{\ell mn},\\
    y &= \frac{(x+2B^k y) -x}{2B^k} = \frac{\ell(m^2-n^2)-\ell(m-n)^2}{2\sqrt{\ell mn}} = \sqrt{\frac{\ell n}{m}}(m-n).
\end{aligned}
\end{equation}

We also check that $B\ge 2$, which holds because $m>n\ge 1$ implies $B=(\ell mn)^{\frac{1}{2k}} >1$, where $B$ is an integer.

To clarify, this shows that the $(a,b,c)$ are in bijection with the $(x,y,B^k)$, and our injection from the $(\ell,m,n)$ to the $(a,b,c)$ in \eqref{eqn:lmn-to-abc-injection} has now led to an injection from the $(\ell,m,n)$ to the $(x,y,B^k)$.

The expression found for $y$ requires that $m\mid \ell n$, but $m$ and $n$ are coprime, so $m\mid \ell$. Then, there exists a unique positive integer $t$ such that $mt = \ell$, and replacing each $\ell$ in \eqref{eqn:preliminary-digital-param} with $mt$ gives $(x,y,B^k)$ in the form of \eqref{eqn:digital-param}.

For the inequalities in \eqref{eqn:lower-digital-ineq} and \eqref{eqn:upper-digital-ineq}, we show their equivalence with $B^{k-1} \le x < B^k$ from \eqref{eqn:unrefined-ineq}:
\begin{align*}
    x < B^k &\iff tm(m-n)^2 < m\sqrt{tn}\\
    &\iff t(m-n)^4 < n
\end{align*}
and
\begin{align*}
    B^{k-1} \le x &\iff \left[(B^{k})^2\right]^{k-1} \le x^{2k}\\
    &\iff [(m\sqrt{tn})^{2}]^{k-1} \le (tm(m-n)^2)^{2k}\\
    &\iff m^{2k-2} t^{k-1} n^{k-1} \le t^{2k} m^{2k} (m-n)^{4k}\\
    &\iff n^{k-1} \le m^2 t^{k+1} (m-n)^{4k}.
\end{align*}

For the converse, we can perform a direct computation to show that $(x,y,B^k)$ in the form of \eqref{eqn:digital-param} causes $(x,y,B,k)$ to satisfy the original \eqref{eqn:main-anomaly}:
\begin{align*}
    y+\frac{x}{B^k} &= \sqrt{tn}(m-n) + \frac{tm(m-n)^2}{m\sqrt{tn}}\\
    &= \frac{tmn(m-n)+tm(m-n)^2}{m\sqrt{tn}}\\
    &= \frac{tm(m-n)(n+(m-n))}{m\sqrt{tn}}\\
    &= \frac{tm^2 (m-n)}{m\sqrt{tn}} = \frac{tm(m-n)^2}{\sqrt{tn}(m-n)} = \frac{x}{y}.
\end{align*}
The inequalities were already shown to imply $B^{k-1} \le x < B^k$, which is the last requirement for $(x,y,B,k)$ to be a digital anomaly.
\end{proof}

We make some comments about Theorem~\ref{thm:anomaly-param}:
\begin{itemize}
    \item It is necessary that $tn$ is a square.
    \item Since the described set of $(t,m,n)$ are in bijection with the set of $(x,y,B^k)$ over all digital anomalies $(x,y,B,k)$, via the correspondence in \eqref{eqn:digital-param}, we will freely refer to the parameters $(t,m,n)$ for any given $(x,y,B^k)$ throughout the rest of the paper.
    \item The result parametrizes $B^k$, but it does not necessarily uniquely determine $B$ or $k$ separately. At least in this approach through Pythagorean triples, it appears difficult to decouple $B$ and $k$, but the following result shows that $k$ is determined by $m,n,$ and $B$.
\end{itemize}

\begin{corollary}\label{crl:k-in-mnB}
For any digital anomaly $(x,y,B,k)$,
\begin{equation}\label{eqn:k-using-mnB}
B^k < \frac{mn}{(m-n)^2} \le B^{k+1}.
\end{equation}
Consequently,
$$k = \left\lceil \log_B \frac{mn}{(m-n)^2} \right\rceil -1.$$
\end{corollary}

\begin{proof}
While figuring out the double inequality in \eqref{eqn:k-using-mnB} was the result of experimentation with manipulations that cannot be recounted here, it can be proven quickly in a backwards manner, since the result has already been articulated:
\begin{align*}
    B^k < \frac{mn}{(m-n)^2} &\iff B^k (m-n)^2 < mn\\
    &\iff m\sqrt{tn}(m-n)^2<mn\\
    &\iff t(m-n)^4 <n,
\end{align*}
which we know from \eqref{eqn:lower-digital-ineq}. Secondly,
\begin{align*}
    \frac{mn}{(m-n)^2} \le B^{k+1} &\iff mn \le B^{k+1}(m-n)^2\\
    &\iff m^{2k} n^{2k} \le [(B^k)^2]^{k+1}(m-n)^{4k}\\
    &\iff m^{2k} n^{2k} \le [(m\sqrt{tn})^2]^{k+1} (m-n)^{4k}\\
    &\iff m^{2k} n^{2k} \le m^{2k+2} t^{k+1}n^{k+1} (m-n)^{2k}\\
    &\iff n^{k-1} \le m^2 t^{k+1} (m-n)^{4k},
\end{align*}
which is \eqref{eqn:upper-digital-ineq}. For the consequence, \cite[p.~69]{Knuth} tells us that
$$\lceil\alpha\rceil = \beta \iff \beta - 1 < \alpha \le \beta,$$
for real numbers $\alpha,\beta$, so
\begin{align*}
     B^k < \frac{mn}{(m-n)^2} \le B^{k+1} &\iff k < \log_B \frac{mn}{(m-n)^2} \le k+1\\
     &\iff \left\lceil \log_B \frac{mn}{(m-n)^2} \right\rceil = k+1.
\end{align*}
This is equivalent to what we sought.
\end{proof}

Corollary~\ref{crl:k-in-mnB} is not immediately useful to us, since $B$ is not parametrized in terms of $m$ and $n$. However, it allows for the formulation of a heuristic argument for why digital anomalies with larger $k$ are increasingly rarer as $k\to \infty$.

The heuristic argument involves considering how $B^k$ might be equal to a different power $C^j$, and how that affects the bounds in \eqref{eqn:k-using-mnB}. First, we try to make $j$ less than $k$. For a given triple $(t,m,n)$ to produce a digital anomaly $(x,y,B,k)$, it is necessary that $B^k = m\sqrt{tn}$ is a perfect $k^{\text{th}}$ power and that \eqref{eqn:k-using-mnB} holds. Then, for any positive integers $C$ and $j$ such that $j<k$ and $C^j = B^k$, since
$$j<k\iff k+1 < k\left(1+\frac{1}{j}\right),$$
we find that
$$C^j = B^k < \frac{mn}{(m-n)^2} \le B^{k+1} < (B^k)^{1+\frac{1}{j}} = (B^k)^{\frac{j+1}{j}} = (B^{\frac{k}{j}})^{j+1} = C^{j+1}.$$

So, every positive integer $j$ such that $B^k = C^j$ and $j<k$ works, which means we can always go to the bottom by letting the exponent be $j=1$.

However, going from $k=1$ to higher exponents $j$ shows that $\frac{mn}{(m-n)^2}$ gets squeezed into a very tight interval, as follows. Suppose $$B^1 < \frac{mn}{(m-n)^2} \le B^2$$
and that $C^j = B^1$ also works. Then
$$B^1 = C^j < \frac{mn}{(m-n)^2} \le C^{j+1} = (B^{\frac{1}{j}})^{j+1} = B^{1+\frac{1}{j}}\le  B^2,$$
where we see that the upper bound of $B^{1+\frac{1}{j}}$ is closer to $\frac{mn}{(m-n)^2}$ than the original $B^2$. As $j$ increases, the interval in which $\log_B \frac{mn}{(m-n)^2}$ has to lie becomes narrower, from $(1,2]$ to $\left(1,1+\frac{1}{j}\right]$.

As a side note, we remark that the highest exponent $k$, such that $B^k = m\sqrt{tn}$ for some integer $B\ge 2$, is the greatest common divisor $g$ of the multiplicities of the prime power factors of $m\sqrt{tn}$. Due to the fact that the common divisors of a set of integers are precisely the factors of their greatest common divisor, we know that $m\sqrt{tn} = C^j$ is a perfect $j^{\text{th}}$ power for some positive integer $C$ if and only if $j$ divides $g$. However, there is no guarantee that $g$ is the highest exponent that will satisfy the inequalities \eqref{eqn:k-using-mnB}.

\section{Infinite Families}

In this section, we look at families of digital anomalies $(x,y,B,k)$ in which $\gcd(x,y)=1$ or $\gcd(x,y) = d$ for each fixed integer $d>1$. The original example of $\frac{5}{2} = 2.5$ is in lowest terms, so it is not yet clear whether there exist examples in which $\frac{x}{y}$ is not in lowest terms.

\begin{theorem}
The set of all digital anomalies $(x,y,B,k)$, such that $x$ and $y$ are coprime, is given by
\begin{equation}\label{eqn:coprime-anomaly-param}
    (x,y,B,k) = (s^2+1,s,(s^2+1)s,1),
\end{equation}
ranging over all integers $s\ge 2$. It is clear that there is no repetition.
\end{theorem}

\begin{proof}
The key is to use our workhorse, equation \eqref{thm:anomaly-param} from Theorem~\ref{thm:anomaly-param}. To deduce the structure of digital anomalies $(x,y,B,k)$ such that $x$ and $y$ are coprime, we let $$(x,y,B^k) = (tm(m-n)^2, \sqrt{tn}(m-n),m\sqrt{tn}),$$
where $t,m,n$ are unique positive integers such that $m>n$ and $m,n$ are coprime.

Suppose $\gcd(x,y) = 1$. Then, $\gcd(x,y^2) = 1$ as well, so
\begin{align*}
    1 &= \gcd(x,y^2) = \gcd(tm(m-n)^2,tn(m-n)^2)\\
    &= t(m-n)^2 \cdot \gcd(m,n) = t(m-n)^2,
\end{align*}
which gives $t=1$ and $m = n+1$. This causes the simplification
\begin{equation}\label{eqn:preliminary-non-coprime-param}
    (x,y,B^k) = (n+1, \sqrt{n}, (n+1)\sqrt{n}).
\end{equation}
Then, $n$ must be a square, so we let $y^2 = n = s^2$ for some integer $s = y \ge 2$, which turns \eqref{eqn:preliminary-non-coprime-param} into
\begin{equation}\label{eqn:near-non-coprime-param}
    (x,y,B^k) = (s^2+1, s, (s^2+1)s).
\end{equation}

We can see that \eqref{eqn:near-non-coprime-param} satisfies \eqref{eqn:main-anomaly} by checking,
$$\frac{x}{y} = \frac{s^2 +1}{s} = s+\frac{1}{s} = s + \frac{s^2 + 1}{(s^2 +1)s} = y + \frac{x}{B^k}.$$
To prove the inequality \eqref{eqn:unrefined-ineq}, we find $k$ first. Since $B^k = s(s^2 +1)$ and $\gcd(s,s^2 +1)$ are coprime, we find that $s$ and $s^2 +1$ must both be perfect $k^{\text{th}}$ powers. Then, $s^2$ and $s^2 +1$ are consecutive positive integers that are both perfect $k^{\text{th}}$ powers, which is impossible for $k\ge 2$ due to the widening gap. So, $k=1$, and we prove \eqref{eqn:unrefined-ineq},
$$B^{k-1} = B^0 = 1 \le x = s^2 + 1 < (s^2 +1)s = B^k.$$ Therefore, the parametrization is complete.
\end{proof}

\begin{example}
The first two examples with $\gcd(x,y)=1$ are:
\begin{itemize}
    \item If $s=2$ is substituted into \eqref{eqn:coprime-anomaly-param}, then $(x,y,B,k)=(5,2,10,1)$. So, the original observation that $\frac{5}{2} = 2.5$ in base $10$ is the first example with coprime $x$ and $y$.
    \item If $s=3$, then $(x,y,B,k) = (10,3,30,1)$. We can write $10_{10}=a_{30}$, since $9<10<30$ means that $10$ is a digit in base $30$ but not in base $10$. Then, the computation works out as
    \begin{align*}
        \left(\frac{a}{3}\right)_{30} = \left(\frac{10}{3}\right)_{10} = 3_{10} + \left(\frac{1}{3}\right)_{10} &= 3_{10}+\left(\frac{10}{30}\right)_{10}\\
        &= 3_{30} + \left(\frac{a}{10}\right)_{30} = (3.a)_{30}.
    \end{align*}
\end{itemize}
\end{example}

Although it is challenging to find \textit{all} digital anomalies $(x,y,B,k)$ where $\gcd(x,y)$ has a fixed value $d>1$, we can parametrize an infinite family of such solutions in each case.

\begin{theorem}\label{thm:gcd-above-1-family}
For each integer $d>1$, an infinite family of digital anomalies $(x,y,B,k)$ such that $\gcd(x,y)=d$ is given by
\begin{equation}\label{eqn:gcd-above-1-family}
    (x,y,B,k) = ((s^2+d)d^2,sd,(s^2+d)s,1),
\end{equation}
where $s$ is any positive integer that is coprime to $d$ and $s\ge d^2 + 1$.
\end{theorem}

\begin{proof}
It is not difficult to verify that this family works, which we will do at the end of the proof. First, we will show how to construct it. Suppose $(x,y,B,k)$ is a digital anomaly. Using the parametrization \eqref{eqn:digital-param} from Theorem \ref{thm:anomaly-param},
\begin{align*}
\gcd(x,y) &= \gcd(tm(m-n)^2,\sqrt{tn}(m-n))\\
&= (m-n)\cdot\gcd(tm(m-n),\sqrt{tn}).
\end{align*}
To get $\gcd(x,y)$ to get closer to becoming $d$, we take $t=1$ and $m-n=d$ because it forces
$$(m-n)\cdot\gcd(tm(m-n),\sqrt{tn}) = d\cdot \gcd(md,\sqrt{n})=d\cdot \gcd(d,\sqrt{n}),$$
due to $m$ and $n$ being coprime.
Now it suffices to let $\sqrt{n}=s$ be an integer that is coprime to $d$, so that
$$\gcd(x,y) = d\cdot \gcd(d,\sqrt{n}) = d\cdot\gcd(d,s)=d.$$
So, using $t=1$, $m=n+d$, and $n=s^2$, \eqref{eqn:digital-param} becomes
\begin{align*}
    (x,y,B^k) &= (tm(m-n)^2,\sqrt{tn}(m-n),m\sqrt{tn})\\
    &= ((s^2+d)d^2,sd,(s^2 +d)s),
\end{align*}
which is \eqref{eqn:gcd-above-1-family} by taking $k=1$.

We can verify that \eqref{eqn:main-anomaly} holds here,
$$\frac{x}{y} = \frac{(s^2 +d)d^2}{sd} = sd + \frac{d^2}{s} = \frac{(s^2 + d)d^2}{(s^2 + d)s} = y + \frac{x}{B^k}.$$

Since we chose $k=1$, this makes the inequality constraint in \eqref{eqn:lower-digital-ineq} into
$$t(m-n)^4 < n \iff d^4 < s^2 \iff d^2 < s \iff d^2 + 1\le s,$$
which is simply a lower bound on $s$ in terms of the fixed $d$. Lastly, we note that \eqref{eqn:upper-digital-ineq} is trivially true when $k=1$ is substituted into it.
\end{proof}

\begin{example}
The smallest case of the families described in Theorem~\ref{thm:gcd-above-1-family} substitutes $d=2$ and $s=2^2+1=5$ into \eqref{eqn:gcd-above-1-family}. This yields
$$(x,y,B,k) = ((s^2+d)d^2,sd,(s^2+d)s,1) = (108,10,135,1).$$
We can check, using base $10$ forms, that
$$\frac{108}{10} = 10 + \frac{4}{5}  = 10 + \frac{108}{135^1}.$$
\end{example}

\section{Finite Classes}

Our final area of exploration will be to look at the set of digital anomalies for a fixed base $B$ or a fixed exponent $k$.

For analyzing fixed bases, the following is a significant lemma that is taken from Alan Baker's famous work in transcendental number theory.

\begin{lemma}\label{lem:Baker-Wustholz}
Let $\alpha_1,\ldots,\alpha_u$ be algebraic numbers, not $0$ or $1$, for some positive integer $u$. Let $\log \alpha_1,\ldots,\log \alpha_u$ denote the principal value of each logarithm. Let $K$ be the field generated by $\alpha_1,\ldots,\alpha_u$ over the rationals $\mathbb{Q}$ and let $d$ be the degree of $K$. For each $\alpha_i$, let $A_i\ge e$ be an upper bound on the absolute values of the relatively prime integer coefficients of the minimal polynomial of $\alpha_i$, where $e$ is Euler's constant. Let $c_1,\ldots,c_u$ be integers, not all $0$, and let $C\ge e$ be an upper bound on $\{|c_1|,\ldots,|c_u|\}$. Denoting
$$\Lambda = c_1\log \alpha_1 + \cdots + c_u \log \alpha_u,$$
it holds that
\begin{equation}\label{eqn:Baker-general-lower}
    \log |\Lambda| > -16(ud)^{2(u+2)} (\log A_1 \cdots \log A_u)\log C,
\end{equation}
which consists of natural logarithms of positive real numbers. \cite[p.~20]{Baker}
\end{lemma}

As a direct consequence, we obtain the following special case for our purposes.

\begin{lemma}\label{lem:special-Baker}
Let $p_1,\ldots,p_u$ be distinct primes. Let $c_1,\ldots,c_u$ be integers, not all $0$, and let $C\ge e$ be an upper bound on the set $\{|c_1|,\ldots,|c_u|\}$. Denoting
$$\Lambda = c_1 \log p_1 +\cdots + c_u \log p_u = \log (p_1^{c_1}\cdots p_u^{c_u}),$$
where all logarithms are real and in the natural base, it holds that
\begin{equation}\label{eqn:Baker-special-lower}
    \log |\Lambda| > -2(16u)^{2(u+2)}(\log p_1 \cdots \log p_u) \log C.
\end{equation}
For future reference, we will use $D$ to denote the positive constant $$D=2(16u)^{2(u+2)}(\log p_1 \cdots \log p_u),$$ which depends only on the $p_i$ and $u$ (the number of $p_i$).
\end{lemma}

\begin{proof}
In Lemma~\ref{lem:Baker-Wustholz}, we take each $\alpha_i$ to be $p_i$, which has the minimal polynomial $x-p_i$, so $A_i = p_i$ is acceptable in all but the case of the prime $2$. If the prime $p_i=2$ appears, then $A_i = p_i = 2 < e$, whereas we are required to have $A_i \ge e$. To get a consistent bound, we always choose the $A_i$ corresponding to the smallest appearing prime $p_i$ to be $A_i = p_i^2$, which is always greater than both $p_i$ and $e$. Then,
$$\log A_i = \log (p_i^2) = 2 \log p_i,$$ which explains our extra factor of $2$ in $D$ in \eqref{eqn:Baker-special-lower}, compared to the expression for the lower bound in \eqref{eqn:Baker-general-lower} in Lemma~\ref{lem:Baker-Wustholz},

The degree of $K$ is just $d=1$, since the primes $p_i$ are all in $\mathbb{Q}$.

Lastly, we can let $C = \max\{e,|c_1|,\ldots,|c_u|\}$, where $e$ is included among the elements, just in case all of the integers $c_i$ are too small by all being in the interval $[-2,2]$, since we require $C\ge e$.
\end{proof}

A variation of this special case of Lemma~\ref{lem:Baker-Wustholz} is mentioned in \cite[p.~1229]{Granville}, but that article appears to have missed the necessity of our extra factor of $2$ in $D$.

\begin{theorem}
For each base $B\ge 2$ there exist finitely many digital anomalies $(x,y,B,k)$.
\end{theorem}

\begin{proof}
Suppose $B\ge 2$ is a fixed integer, and that $(x,y,B,k)$ is a digital anomaly. By Theorem~\ref{thm:anomaly-param}, there exist positive integers $(t,m,n)$ in terms of which $x,y$ can be written and $$B^k = m\sqrt{tn} \implies B^{2k} = m^2 nt.$$
So, the prime factors of $m$ and $n$ are all from among the prime factors of $B$. Moreover, it is known that $m,n$ are coprime, so the set of prime factors of $m$ are disjoint from the set of prime factors of $n$. Let the prime factorizations of $B,m,$ and $n$ be $r_1^{c_1}\cdots r_u^{c_u}$, $p_1^{a_1}\cdots p_v^{a_v}$, and  $q_1^{b_1}\cdots q_w^{b^w}$, respectively, where the primes in each individual prime factorization are distinct, and the multiplicities are all non-zero.

The general hope is to put a bound above each of $t,m,n$ in terms of a constant depending only on $B$, regardless of $k$, so that there are only finitely many possibilities remaining. From \eqref{eqn:lower-digital-ineq} and $m>n$, we know that
\begin{align}
    t \le t(m-n)^4 < n &\implies t < n,\\
    (m-n)^4 \le t(m-n)^4 < n &\implies m < n^{\frac{1}{4}} + n \le 2n. \label{eqn:m-n-ineq}
\end{align}
Combining these yields the chain $$t < n < m < 2n.$$
This means  that it suffices to bound $m$ or $n$ in terms of a constant depending only on $B$ because, then, the same will follow for the other two parameters from the triple $(t,m,n)$. This strategy is achievable using Lemma~\ref{lem:special-Baker}, as it will be shown.

To get an expression $\Lambda$ that is a linear combination of logarithms, with integer coefficients, we consider
\begin{equation}\label{eqn:Lambda}
    \log \left(\frac{m}{n}\right) = \log m - \log n = \sum_{i=1}^{v}{a_i\log p_i} + \sum_{j=1}^{w}{(-b_j)\log q_i},
\end{equation}
which is desirable because it has the multiplicities of the primes factor of $m$ and $n$ as the coefficients of the logarithms, and it suffices to bound them. Since $m>n$, the expression in \eqref{eqn:Lambda}, which we will call $\Lambda$, is positive. Using the well-known fact that $\log (x+1) < x$ for all $x>0$, we get
$$|\Lambda| = \Lambda = \log\left(\frac{m}{n}\right) = \log\left(1+\frac{m-n}{n}\right) < \frac{m-n}{n}.$$
Then, \eqref{eqn:m-n-ineq} tells us that
$$m-n < n^{\frac{1}{4}} \implies \frac{m-n}{n} < n^{-\frac{3}{4}}.$$
Therefore, we have the elementary upper bound
$$\log |\Lambda| < \log \frac{m-n}{n} < \log(n^{-\frac{3}{4}}) = -\frac{3}{4}\log n.$$

On the other hand, Lemma~\ref{lem:special-Baker} gives the lower bound $$\log |\Lambda| > - D \log C,$$
where
\begin{align*}
    D &= 2(16(v+w))^{2(v+w+2)} (\log p_1 \cdots \log p_v \cdot \log q_1 \cdots \log q_w),\\
    C &= \max \{e, a_1,\ldots,a_v,b_1\ldots,b_w\}.
\end{align*}
In order to make $D$ depend purely on $B$, we replace it with the larger number,
$$D_B = 2 (16u)^{2(u+2)} (\log r_1\cdots\log r_u),$$
which is acceptable since $u\ge v+w$, and $\{p_1,\ldots,p_v\}$ and $\{q_1,\ldots,q_w\}$ are disjoint subsets of the set of prime factors $\{r_1,\ldots,r_u\}$ of $B$. Then,
$$\log |\Lambda| > -D\log C > -D_B \log C.$$
Combining the upper and lower bound for $\log |\Lambda|$ creates our fundamental squeeze:
\begin{equation}\label{eqn:fundamental-squeeze}
    -D_B \log C < \log |\Lambda| < -\frac{3}{4}\log n \implies \log n < \frac{4}{3} D_B \log C.
\end{equation}
While this is an upper bound on $n$, $C$ still depends on the $a_i$ and the $b_j$, which are not yet bounded. There are three (possibly overlapping) cases of $C$: it is $e$ or it is $\max\{b_1,\ldots, b_w\}$ or it is $\max\{a_1,\ldots,a_v\}$. These are handled as follows:
\begin{enumerate}
    \item If $C=e$, then \eqref{eqn:fundamental-squeeze} implies
    $$\log n < \frac{4}{3}  D_B \log e = \frac{4}{3} D_B \implies n = e^{\log n} < e^{\frac{4}{3}D_B},$$
    which is an upper bound on $n$, in terms of a constant that depends only on $B$.
    \item Suppose $C=\max\{b_1,\ldots,b_w\}=b_j$ for some index $j$. Letting $p$ be the smallest prime factor of $B$, we get $$n \ge q_j^{b_j}\ge p^{b_j} = p^C,$$
    which leads to $$C\log p = \log(p^C)\le \log n < \frac{4}{3}D_B \log C \implies \frac{C}{\log C} < \frac{4D_B}{3\log p}.$$
    This proves that $\frac{C}{\log C}$ is bounded above in terms of a constant, assuming $B$ is fixed. Recall that, over real $z$,
    $$\frac{d}{dz} \frac{z}{\log z} = \frac{\log z -1}{(\log z)^2}$$
    is positive for all real $z>e$, so $\frac{z}{\log z}$ is strictly increasing there. We know that $C$ is an integer above $e$ in this case, which implies the existence of an upper bound $W_B$ for $C$, using only $B$. Finally, \eqref{eqn:fundamental-squeeze} produces
    $$n = e^{\log n} < e^{\frac{4}{3}D_B \log C} = C^{\frac{4}{3} D_B}\le W_B^{\frac{4}{3} D_B},$$
    which is a constant upper bound for $n$, as long as $B$ is known.
    \item Suppose $C=\max\{a_1,\ldots,a_v\}=a_i$ for some index $i$. In this case, we use \eqref{eqn:m-n-ineq} to get
    \begin{align}
        m < 2n &\implies \log m < \log 2 + \log n < \log 2 + \frac{4}{3}D_B\log C\\
        &\implies \log \frac{m}{2} = \log m -\log 2 < \frac{4}{3}D_B\log C.\label{eqn:m-upper-bound}
    \end{align}
    Again, we let $p$ be the smallest prime factor of $B$. Similar to the second case, we get $m\ge p^C$, so
    $$C\log \frac{p}{2} \le \log \frac{p^C}{2^C} \le \log \frac{p^C}{2} \le \log \frac{m}{2} < \frac{4}{3} D_B\log C,$$
    giving
    $$\frac{C}{\log C} < \frac{4D_B}{3\log \frac{p}{2}},$$
    where the upper bound is a constant when $B$ is fixed. As before, the strictly increasing property of $\frac{z}{\ln z}$ for real $z>e$ yields an implicit upper bound $V_B$ for $C$ that depends only on $B$, so \eqref{eqn:m-upper-bound} produces
    $$m = 2\cdot \frac{m}{2} = 2e^{\log \frac{m}{2}}< 2e^{\frac{4}{3}D_B\log C} = 2C^{\frac{4}{3}D_B} \le 2V_B^{D_B},$$
    which is a constant upper bound for $m$ if $B$ is known.
\end{enumerate}
The bounds in each case are large, so calculations on the ``finitely many'' cases would still be computationally intensive.
\end{proof}

Lastly, we apply the $abc$ conjecture to the case of fixed exponents $k$ to derive a conditional result.

\begin{conjecture}[$abc$ conjecture]
Let $\epsilon >0$ be a fixed real number. Then, the set of triples $(a,b,c)$ of coprime positive integers such that $a+b=c$ with $a<b$ and $$\text{Rad}(abc) < c^{\frac{1}{1+\epsilon}}$$
is finite. Here, $\text{Rad}(\gamma)$ denotes the radical of the positive integer argument $\gamma$, which is the product of the distinct prime factors of $\gamma$. \cite[p.~213]{Waldschmidt}.
\end{conjecture}

\begin{theorem}
Assuming the $abc$ conjecture, the collection of digital anomalies $(x,y,B,k)$ is finite for each fixed $k\ge 3$.
\end{theorem}

\begin{proof}
Suppose $(x,y,B,k)$ is a digital anomaly for some $k\ge 3$. By Theorem~\ref{thm:anomaly-param}, we know that positive integers $t,m,n$ exist such that
$$B^k = m\sqrt{tn} \implies B^{2k} = m^2 tn \implies mn\mid B^k,$$
so the set of prime factors of $mn$ is a subset of that of $B$. Moreover, recall from \eqref{eqn:m-n-ineq} that
$$m < 2n \implies m-n <n.$$
This allows us to set $(a,b,c) = (m-n,n,m)$ in the $abc$ conjecture. Then,
\begin{equation}\label{eqn:abc-Rad-initial}
\begin{aligned}
    \text{Rad}((m-n)nm)&\le \text{Rad}(B^k (m-n))\\
    &= \text{Rad}(B(m-n))\le B(m-n),
\end{aligned}
\end{equation}
since powers in the argument of the radical function are irrelevant. Now, we obtain the following estimates for $B$ using \eqref{eqn:k-using-mnB} and $m-n$ using \eqref{eqn:m-n-ineq}, along with $n<m$,
\begin{align*}
    B^k \le (m-n)^2 B^k < mn < m^2 \implies & B < m^{\frac{2}{k}},\\
    &m-n < n^{\frac{1}{4}} < m^{\frac{1}{4}}.
\end{align*}
Together, for $k\ge 3$, they yield
$$B(m-n) < m^{\frac{2}{k}} m^{\frac{1}{4}} \le m^{\frac{2}{3}+\frac{1}{4}} = m^{\frac{11}{12}},$$
Since $\frac{11}{12} = \frac{1}{\left(\frac{12}{11}\right)} = \frac{1}{\left(1+\frac{1}{11}\right)}$, we can take $\epsilon = \frac{1}{11}$ to get
$$\text{Rad}((m-n)nm) < m^{\frac{1}{1+\epsilon}}.$$
By the $abc$ conjecture, there are finitely many triples for $\epsilon = \frac{1}{11}$. This means the set of pairs $(m,n)$ produced from the set of digital anomalies $(x,y,B,k)$ is finite for each fixed $k\ge 3$, which bounds $m$ and $n$ above by a constant using only $k$. Since $0<t<n$, we know that $t$ is bounded above as well, so the set of triples $(t,m,n)$ is finite. By Theorem~\ref{thm:anomaly-param}, $x$ and $y$ can be written in terms of $(t,m,n)$ and $B = (m\sqrt{tn})^{\frac{1}{k}}$ is also determined when $k$ is fixed.

Therefore, contingent on the truth of the $abc$ conjecture, fixing $k\ge 3$ yields a finite set of digital anomalies.
\end{proof}

Unfortunately, applying the $abc$ conjecture to the $k=2$ case does not appear to be feasible, as it is difficult to find a suitable constant $\epsilon > 0$. We can come close with an argument based on \eqref{eqn:k-using-mnB} and \eqref{eqn:abc-Rad-initial}:
\begin{align*}
    B^k < \frac{mn}{(m-n)^2} &\implies  B^2 (m-n)^2 < mn < m^2\\
    &\implies \text{Rad}((m-n)nm) \le B(m-n) < m.
\end{align*}
This is equivalent to taking $\epsilon =0$ in the exponent $\frac{1}{1+\epsilon}$ of $m$ in the upper bound, whereas $\epsilon >0$ remains elusive. To that end, we offer a conjecture for $k=2$.

\begin{conjecture}
The only digital anomalies $(x,y,B,k)$ for which $k=2$ are $(18,4,6,2)$ and $(1323,36,42,2)$.
\end{conjecture}

In fact, the author is unaware of any other examples for which $k>1$, so these might be the only solutions outside $k=1$. This possibility leaves the door open for stronger results than those derived here.

\end{document}